\documentclass[
submission
]{dmtcs-episciences}

\usepackage{latexsym}
\usepackage{amsmath}

\newcommand{\qdn}{\hspace*{-1.5mm}}
\newcommand{\qqdn}{\hspace*{-2.5mm}}
\newcommand{\xqdn}{\hspace*{-5.0mm}}
\newcommand{\xxqdn}{\hspace*{-10mm}}

\newcommand{\fns}{\footnotesize}
\newcommand{\sst}{\scriptstyle}

\newcommand{\ffnk}[4]{\left[\qdn\ba{#1}#3\\#4\ea{\!\Big|\:#2}\right]}

\newcommand{\binm}{\binom}

\newcommand{\nnm}{\nonumber}
\newcommand{\be}{\begin{equation}}
\newcommand{\ee}{\end{equation}}
\newcommand{\ba}{\begin{array}}
\newcommand{\ea}{\end{array}}
\newcommand{\bmn}{\begin{eqnarray}}
\newcommand{\emn}{\end{eqnarray}}
\newcommand{\bnm}{\begin{eqnarray*}}
\newcommand{\enm}{\end{eqnarray*}}
\newcommand{\bln}{\begin{subequations}}
\newcommand{\eln}{\end{subequations}}

\newcommand{\lam}{\lambda}

\newtheorem{thm}{Theorem}
\newtheorem{lemm}[thm]{Lemma}
\newtheorem{corl}[thm]{Corollary}

\newtheorem{exam}{Example}

\usepackage[utf8]{inputenc}
\usepackage{subfigure}

\usepackage[round]{natbib}

\author{Chuanan Wei\affiliationmark{1}\thanks{He is supported by the National Natural Science Foundation of China (No. 11661032).}
  \and Lily Li Liu\affiliationmark{2}\thanks{She is supported by the National Natural Science Foundation of China (No. 11871304).}
 \and Dianxuan Gong\affiliationmark{3}\thanks{I am the corresponding author and am supported by the National Natural Science Foundation of China (No. 11601151).}
  }
\title[Summation formulas for Fox-Wright function]{Summation formulas for Fox-Wright function}
\affiliation{
   Department of Medical Informatics,
 Hainan Medical University, Haikou, China\\
  School of Mathematical Sciences,
 Qufu Normal University, Qufu, China\\
 College of Sciences,
  North China University of Science and Technology, Tangshan, China}
\received{2008-05-01} \revised{2018-09-19}
\accepted{2018-09-19}
\begin{document}
\publicationdetails{20}{2018}{2}{9}{4476}
 \maketitle
\begin{abstract}
   By means of inversion techniques and several known hypergeometric
series identities, summation formulas for Fox-Wright function are
explored. They give some new hypergeometric series identities when
the parameters are specified.
\end{abstract}

\section{Introduction}
 For a complex variable $x$ and an nonnegative integer $n$, define
the shifted-factorial to be
\[(x)_0=1
\quad\text{and}\quad (x)_n=x(x+1)\cdots(x+n-1) \quad\text{when}\quad
n=1,2,\cdots.\]
  Following \cite{andrews-r}, define the hypergeometric series
  by
\[_{1+r}F_s\ffnk{cccc}{z}{a_0,&a_1,&\cdots,&a_r}{&b_1,&\cdots,&b_s}
 \:=\:\sum_{k=0}^\infty
\frac{(a_{0})_{k}(a_{1})_{k}\cdots(a_{r})_{k}}
 {k!(b_{1})_{k}\cdots(b_{s})_{k}}z^k,\]
where $\{a_{i}\}_{i\geq0}$ and $\{b_{j}\}_{j\geq1}$ are complex
parameters such that no zero factors appear in the denominators of
the summand on the right hand side. Then Whipple's $_3F_2$-series
identity (cf. \cite{andrews-r}[p. 149]) can be stated as
 \bmn\label{whipple}
 _3F_2\ffnk{cccc}{1}{a,1-a,b}{c,1+2b-c}
=\frac{\Gamma(\frac{c}{2})\Gamma(\frac{1+c}{2})\Gamma(b+\frac{1-c}{2})\Gamma(b+\frac{2-c}{2})}
{\Gamma(\frac{a+c}{2})\Gamma(\frac{1-a+c}{2})\Gamma(b+\frac{1+a-c}{2})\Gamma(b+\frac{2-a-c}{2})},
 \emn
where $Re(b)>0$ and $\Gamma(x)$ is the well-known gamma function
\[\xqdn\Gamma(x)=\int_{0}^{\infty}t^{x-1}e^{-t}dt\quad\text{with}\quad
Re(x)>0.\] About the research for the generalizations of Watson's
$_3F_2$-series identity and \eqref{whipple} with integer parameters,
the reader may refer to the papers
\cite{chu-b,chu-c,lavoie-a,lavoie-b,wimp,zeilberger}. The
corresponding $q$-analogues can be found in \cite{wei-b}.  Some
strange evaluations of hypergeometric series can be seen in the
papers \cite{gessel,wang-a,wang-b,wang-c}.

Recall Fox-Wright function $_p\Psi_q$
(cf.~\cite{fox,wright-a,wright-b}; see also \cite{srivastava-a}[p.
21], which is defined by
 \bnm
 {_p\Psi_q}\ffnk{ccc}{z}
{(\alpha_1;A_1),\:\:\cdots,&\!(\alpha_p;A_p)}
{(\beta_1;B_1),\:\:\cdots,&\!(\beta_q;B_q)}
 =\sum_{k=0}^{\infty}\frac{\prod_{i=1}^p\Gamma(\alpha_i+A_ik)}
 {\prod_{j=1}^q\Gamma(\beta_j+B_jk)}\frac{z^k}{k!}
\enm
 and regarded as a generalization of hypergeometric series,
 where the coefficients $\{A_i\}_{i\geq1}$ and $\{B_j\}_{j\geq1}$ are
positive real numbers such that
\[1+\sum_{j=1}^qB_j-\sum_{i=1}^pA_i\geq0.\]
Naturally, the domain where $\{A_i\}_{i\geq1}$ and
$\{B_j\}_{j\geq1}$ take on values can be extended to the complex
field and the upper coefficient relation is ignored when the series
is terminating. For convenience, we shall frequently use the symbol
\[(x_1,x_2,\cdots,x_r;A):=\prod_{i=1}^r\Gamma(x_i+Ak)\]
in the Fox-Wright function.

The importance of Fox-Wright function lies in that it can be applied
to many fields. \cite{miller-a} offered the applications of
Fox-Wright function to the solution of algebraic trinomial equations
and to a problem of information theory. \cite{mainardi} told us that
Fox-Wright function plays an important role in finding the
fundamental solution of the factional diffusion equation of
distributed order in time. More applications of the function can be
found in
\cite{craven,darus,miller-b,murug,pogany,srivastava-b,srivastava-c}.
 The following fact should be
mentioned. \cite{aomoto} introduced the quasi hypergeometric
function, which is exactly multiple Fox-Wright function, and showed
that it satisfies a system of difference-differential equations.

One pair of inverse relations implied in the works of
\cite{bressoud} and \cite{gasper} (see also \cite{chu-a}[p. 17]) can
be expressed as follows.

\begin{lemm}\label{lemm-a}
Let $x,y$ and $z$ be complex numbers. Then the system of equations
 \bmn\label{inver-a}
  \quad f(n)=\sum_{k=0}^{n}(-1)^k\binm{n}{k}
  \frac{x+zk+k}{(x+zn)_{1+k}}
 \frac{y-zk+k}{(y-zn)_{1+k}} (z^{-1}(x-y)+k)_n
 g(k)
  \emn
is equivalent to the system of equations
 \bmn\label{inver-b}
\:g(n)=\sum_{k=0}^{n}(-1)^k\binm{n}{k}(x+zk)_n(y-zk)_n
\frac{z^{-1}(x-y)+2k}{(z^{-1}(x-y)+n)_{1+k}} f(k).
 \emn
\end{lemm}

Although the importance of Fox-Wright function has been realized for
many years, there are only a small number of summation formulas for
this function to our knowledge. The corresponding results, which can
be seen in the papers \cite{chu-d,krattenthaler,wei-a}, are all from
inversion techniques and identities related to Saalschutz's theorem.
 The reader may refer to \cite{ma-a,ma-b,warnaar} for more
details on inversion techniques.

Inspired by the importance of Fox-Wright function and the lack of
summation formulas for this function, we shall derive several
summation formulas for Fox-Wright function according to Lemma
\ref{lemm-a}, \eqref{whipple} and some other hypergeometric series
identities in Sections 2-3.

\section{Whipple-type series and summation formulas for Fox-Wright function}

\begin{thm}\label{thm-a}
Let $a$ and $\lambda$ be complex numbers. Then
 \bnm
&&\xqdn{_5\Psi_6}\ffnk{ccccccc}{-1}
{\fns{(1,\frac{3}{2};1),(\lambda-a;\lambda),(\frac{1}{2}+\lambda-a;1\!+\!\lambda),(1+2a+n;1\!+\!2\lambda)}}
{\fns(1+n;-1),(2+n,\frac{1}{2};1),(1+a;\lambda),(\frac{3}{2}+a;1\!+\!\lambda),(1+2\lambda\!-\!2a\!-\!n;1\!+\!2\lambda)}\\
 &&\xqdn\:\:=\:\:
\frac{(-1)^n4^{2a-\lambda}(1+2a-\lambda)_n}
 {(\lambda-a+\lambda n)(1+2a+2\lambda n+2n)(1)_n}.
 \enm
\end{thm}

\begin{proof}
The case $a=-n$ of \eqref{whipple} reads
 \bnm
 _3F_2\ffnk{cccc}{1}{-n,1+n,b}{c,1+2b-c}
=\frac{(\frac{c-n}{2})_n(c-2b)_n}{(\frac{c-n}{2}-b)_n(c)_n}.
 \enm
Perform the replacements $b\to 1+2a-\lambda$ and $c\to2+2a+2\lambda
n+n$ to obtain
 \bnm
&&\sum_{k=0}^n(-1)^k\:\binm{n}{k}
 \frac{1+2a+2\lam k+k+k}{(1+2a+2\lam n+n)_{1+k}}
 \frac{2a-2\lam-2\lam k-k+k}{(2a-2\lam-2\lambda n-n)_{1+k}}(1+k)_n\\
 &&\:\:\times\:
  \frac{(1)_k(1+2a-\lam)_k}{(1+2a+2\lam k+2k)(2a-2\lam-2\lam k)}\\
&& \:\:=\:\frac{(1)_n}{2a-2\lam-2\lambda n-n} \frac{(1+a+\lambda
n)_n(2\lam-2a+2\lambda n+n)_n}{(\lam-a+\lambda n)_n(1+2a+2\lam
n+n)_{1+n}},
 \enm
which fits to \eqref{inver-a} with
 \bnm
&&\xxqdn x=1+2a,\quad y=2a-2\lam, \quad z=1+2\lam; \\
&&\xxqdn g(k)=\frac{(1)_k(1+2a-\lam)_k}{(1+2a+2\lam
k+2k)(2a-2\lam-2\lam k)};\\
&&\xxqdn f(n)=\frac{(1)_n}{2a-2\lam-2\lambda n-n} \frac{(1+a+\lambda
n)_n(2\lam-2a+2\lambda n+n)_n}{(\lam-a+\lambda n)_n(1+2a+2\lam
n+n)_{1+n}}.
 \enm
Then \eqref{inver-b} produces the dual relation
 \bnm
&&\xqdn\sum_{k=0}^n(-1)^k\binm{n}{k}(1+2a+2\lam k+k)_n(2a-2\lam-2\lambda k-k)_n\frac{1+2k}{(1+n)_{1+k}}\\
 &&\xqdn\:\times\:\,
  \frac{(1)_k}{2a-2\lam-2\lambda k-k} \frac{(1+a+\lambda k)_k(2\lam-2a+2\lambda k+k)_k}
  {(\lam-a+\lambda k)_k(1+2a+2\lam k+k)_{1+k}}\\
 &&\xqdn\:=\:\frac{(1)_n(1+2a-\lam)_n}{(1+2a+2\lam n+2n)(2a-2\lam-2\lam n)}.
 \enm
Rewrite it in accordance with Fox-Wright function, we get Theorem
\ref{thm-a}.
\end{proof}

 When $\lambda=0$, Theorem
\ref{thm-a} reduces to the hypergeometric series identity
 \bnm\quad
  _5F_4\ffnk{ccccc}{1}{1,\frac{3}{2},\frac{1}{2}-a,1+2a+n,-n}{\frac{1}{2},\frac{3}{2}+a,1-2a-n,2+n}
  =\frac{1+2a}{1+2a+2n}\frac{(2)_n}{(2a)_n},
 \enm
which is a special case of  Dougall's $_5F_4$-series identity (cf.
\cite{andrews-r}[p. 71]):
 \bmn\label{dougall}
 &&\xqdn_5F_4\ffnk{cccc}{1}{a,1+\frac{a}{2},b,c,d}{\frac{a}{2},1+a-b,1+a-c,1+a-d}
  \nnm\\
&&\xqdn\:\:=\:\frac{\Gamma(1+a-b)\Gamma(1+a-c)\Gamma(1+a-d)\Gamma(1+a-b-c-d)}
{\Gamma(1+a)\Gamma(1+a-b-c)\Gamma(1+a-b-d)\Gamma(1+a-c-d)}
 \emn
provided that $Re(1+a-b-c-d)>0$.

 Other two hypergeometric series identities from Theorem \ref{thm-a} can be
displayed as follows.

\begin{exam}[$\lambda=1$ in Theorem \ref{thm-a}]
 \bnm
&&\xqdn_9F_8\ffnk{ccccccccc}{1}
{1,\frac{3}{2},1-a,\frac{3-2a}{4},\frac{5-2a}{4},\frac{1+2a+n}{3},\frac{2+2a+n}{3},\frac{3+2a+n}{3},-n}
{\frac{1}{2},1+a,\frac{5+2a}{4},\frac{3+2a}{4},\frac{5-2a-n}{3},\frac{4-2a-n}{3},\frac{3-2a-n}{3},2+n}\\
  &&\xqdn\:\:=\:\frac{2a(1-a)(1+2a)}{(1-a-n)(2a+n)(1+2a+4n)}\frac{(2)_n}{(2a-2)_n}.
 \enm
\end{exam}

\begin{exam}[$\lambda=2$ in Theorem \ref{thm-a}]
 \bnm
&&\qqdn\xxqdn_{13}F_{12}\ffnk{ccccccccccc}{1}
{1,\frac{3}{2},\frac{2-a}{2},\frac{3-a}{2},\frac{5-2a}{6},\frac{7-2a}{6},\frac{9-2a}{6},\{\frac{i+2a+n}{5}\}_{i=1}^5,-n}
{\frac{1}{2},\frac{2+a}{2},\frac{1+a}{2},\frac{7+2a}{6},\frac{5+2a}{6},\frac{3+2a}{6},\{\frac{10-i-2a-n}{5}\}_{i=1}^5,2+n}\\
  &&\qdn\xxqdn\:\:=\:\frac{(2-a)(1+2a)}{(2-a+2n)(1+2a+6n)}\frac{(2a-1)_n(2)_n}{(2a-4)_n(2a+1)_n}.
 \enm
\end{exam}

\begin{thm}\label{thm-b}
Let $a$ and $\lambda$ be complex numbers. Then
 \bnm
\:\:\sum_{i=0}^m(-1)^i\frac{(-m)_i}{i!}\Omega(\lambda,a,i,m,n)=
 \frac{(-1)^{n}2^{4a-2\lambda-m-1}(1+2a-\lambda-\frac{m}{2})_n}{(1+2a+2\lambda n+2n)(\lambda-a+\lambda n)(1)_n},
 \enm
where the expression on the left hand side is
 \bnm
&&\xxqdn\Omega(\lambda,a,i,m,n)\\
 &&\xxqdn\:={_7\Psi_8}\ffnk{ccccccc}{-1}
  {(1,\frac{3}{2};1),(\frac{m-i}{2}+\lambda-a;\lambda),(\frac{1+m-i}{2}+\lambda-a;1+\lambda),\\
(\frac{3+m}{2}-i+\lambda,1+2a+n;1+2\lambda),(1-i+2\lambda;2+4\lambda)}
{(1+n;-1),(2+n,\frac{1}{2};1),(\frac{2-i}{2}+a;\lambda),(\frac{3-i}{2}+a;1+\lambda),\\
(\frac{1+m}{2}-i+\lambda,1+2\lambda-2a-n;1+2\lambda),(2+m-i+2\lambda;2+4\lambda)}.
\enm
\end{thm}

Remark: In the symbol $\Omega(\lambda,a,i,m,n)$, the parameters on
the first two lines are numerators and other ones are denominators.
This remark is also applicative to Theorem \ref{thm-c}.

\begin{proof}
Lemma 1 of \cite{chu-c} gives
 \bnm
&&\xxqdn{_3F_2}\ffnk{ccccccc}{1}{a,1-a,b}{c,1+2b-c+m}=\frac{(1+2b-c)_m}{(2+2b-2c)_m}
\sum_{i=0}^m(-1)^i\frac{(-m)_i(\frac{3}{2}+b-c)_i}{i!(\frac{1}{2}+b-c)_i}
\\
&&\xxqdn\:\:\times\:\:\frac{(1-c)_i(1+2b-2c)_i}{(1+2b-c)_i(2+2b-2c+m)_i}
{_3F_2}\ffnk{ccccccc}{1}{a,1-a,b}{c-i,1+2b-c+i}. \enm Calculating
the $_3F_2$-series on the right hand side by \eqref{whipple} and
then setting $a=-n$ in
 the resulting identity, we have
\bnm
&&\xxqdn{_3F_2}\ffnk{ccccccc}{1}{-n,1+n,b}{c,1+2b-c+m}=\frac{(1+2b-c)_m}{(2+2b-2c)_m}
\sum_{i=0}^m(-1)^i\frac{(-m)_i(\frac{3}{2}+b-c)_i}{i!(\frac{1}{2}+b-c)_i}
\\
&&\xxqdn\:\:\times\:\:\frac{(1-c)_i(1+2b-2c)_i}{(1+2b-c)_i(2+2b-2c+m)_i}
\frac{(\frac{c-n-i}{2})_n(c-2b-i)_n}{(\frac{c-n-i}{2}-b)_n(c-i)_n}.
\enm
 Employ the substitutions $b\to 1+2a-\lambda-\frac{m}{2}$ and
$c\to 2+2a+2\lambda n+n$ to gain
 \bnm
&&\xxqdn\sum_{k=0}^n(-1)^k\:\binm{n}{k}
 \frac{1+2a+2\lam k+k+k}{(1+2a+2\lam n+n)_{1+k}}
 \frac{2a-2\lam-2\lam k-k+k}{(2a-2\lam-2\lambda n-n)_{1+k}}(1+k)_n\\
 &&\xxqdn\:\:\times\:
  \frac{(1)_k(1+2a-\lam-\frac{m}{2})_k}{(1+2a+2\lam k+2k)(2a-2\lam-2\lam k)}\\
&&\xxqdn \:\:=\:\frac{(1)_n}{(1+2a+2\lambda n+n)(2a-2\lam-2\lambda
n-n)}
 \frac{(2\lam-2a+2\lambda n+n)_m}{(1+2\lambda+4\lam n+2n)_m}\\
&&\xxqdn\:\:\times\:\sum_{i=0}^m(-1)^i\frac{(-m)_i(\frac{1-m}{2}-\lambda-2\lambda
n-n)_i(-1-2\lambda-4\lambda n-2n-m)_i}
{i!(-\frac{1+m}{2}-\lambda-2\lambda n-n)_i(1+2a-2\lambda-2\lambda n-n-m)_i}\\
&&\xxqdn\:\:\times\:\frac{(-1-2a-2\lambda n-n)_i}
{(-2\lambda-4\lambda n-2n)_i} \frac{(a+\lambda
n+\frac{2-i}{2})_n(2\lambda-2a+2\lambda
n+n+m-i)_n}{(\lambda-a+\lambda n+\frac{m-i}{2})_n(2+2a+2\lambda n
+n-i)_n},
 \enm
which suits to \eqref{inver-a} with
 \bnm
&&\qdn\xqdn x=1+2a,\quad y=2a-2\lam, \quad z=1+2\lam; \\
&&\qdn\xqdn g(k)=\frac{(1)_k(1+2a-\lam-\frac{m}{2})_k}{(1+2a+2\lam
k+2k)(2a-2\lam-2\lam k)};\\
&&\qdn\xqdn f(n)=\frac{(1)_n}{(1+2a+2\lambda n+n)(2a-2\lam-2\lambda
n-n)}
 \frac{(2\lam-2a+2\lambda n+n)_m}{(1+2\lambda+4\lam n+2n)_m}\\
&&\:\times\:\sum_{i=0}^m(-1)^i\frac{(-m)_i(\frac{1-m}{2}-\lambda-2\lambda
n-n)_i(-1-2\lambda-4\lambda n-2n-m)_i}
{i!(-\frac{1+m}{2}-\lambda-2\lambda n-n)_i(1+2a-2\lambda-2\lambda n-n-m)_i}\\
&&\:\times\:\frac{(-1-2a-2\lambda n-n)_i} {(-2\lambda-4\lambda
n-2n)_i} \frac{(a+\lambda n+\frac{2-i}{2})_n(2\lambda-2a+2\lambda
n+n+m-i)_n}{(\lambda-a+\lambda n+\frac{m-i}{2})_n(2+2a+2\lambda n
+n-i)_n}.
 \enm
Then \eqref{inver-b} offers the dual relation
 \bnm
&&\qdn\sum_{k=0}^n(-1)^k\binm{n}{k}(1+2a+2\lam k+k)_n(2a-2\lam-2\lambda k-k)_n\frac{1+2k}{(1+n)_{1+k}}\\
 &&\qdn\:\times\:\,
 \frac{(1)_k}{(1+2a+2\lambda k+k)(2a-2\lam-2\lambda k-k)}
 \frac{(2\lam-2a+2\lambda k+k)_m}{(1+2\lambda+4\lam k+2k)_m}\\
&&\qdn\:\times\:\sum_{i=0}^m(-1)^i\frac{(-m)_i(\frac{1-m}{2}-\lambda-2\lambda
k-k)_i(-1-2\lambda-4\lambda k-2k-m)_i}
{i!(-\frac{1+m}{2}-\lambda-2\lambda k-k)_i(1+2a-2\lambda-2\lambda k-k-m)_i}\\
&&\qdn\:\times\:\frac{(-1-2a-2\lambda k-k)_i} {(-2\lambda-4\lambda
k-2k)_i}
\frac{(a+\lambda k+\frac{2-i}{2})_k(2\lambda-2a+2\lambda k+k+m-i)_k}{(\lambda-a+\lambda k+\frac{m-i}{2})_k(2+2a+2\lambda k+k-i)_k}\\
&&\qdn\:=\:\frac{(1)_n(1+2a-\lam-\frac{m}{2})_n}{(1+2a+2\lam
n+2n)(2a-2\lam-2\lam n)}.
 \enm
Interchanging the summation order, we achieve Theorem \ref{thm-b}
after some simplifications.
\end{proof}

When $m=0$, Theorem \ref{thm-b} reduces to Theorem \ref{thm-a}.
Performing the the replacements $\lambda\to \lambda-\frac{1}{2}$ and
$a\to\frac{\lambda+a-1}{2}$
 in the case $m=1$ of Theorem \ref{thm-b}, we attain the following
reciprocal formula after some reformulations.

\begin{corl}\label{corl-a}
Let $a$ and $\lambda$ be complex numbers. Then
 \bnm\qquad
A(\lambda;a,n)-A(-\lambda;a,n)=\frac{(-1)^{n+1}2^{2a+1}\,\lambda\,(a)_n}{(a+\lambda+2\lambda
n+n)(a-\lambda-2\lambda n+n)(1)_n},
 \enm
where the symbol on the left hand side stands for
 \bnm
&&\xqdn A(\lambda;a,n)\\
&&\xqdn={_4\Psi_5}\ffnk{ccccccc}{-1}
  {(1;1),(\frac{1+\lambda-a}{2};\lambda-\frac{1}{2}),(\frac{2+\lambda-a}{2};\lambda+\frac{1}{2}),(\lambda+a+n;2\lambda)}
{(1\!+\!n;-1),(2\!+\!n;1),(\frac{1+\lambda+a}{2};\lambda\!-\!\frac{1}{2}),(\frac{2+\lambda+a}{2};\lambda\!+\!\frac{1}{2}),(1\!+\!\lambda\!-\!a\!-\!n;2\lambda)}.
\enm
\end{corl}

Two hypergeometric series identities from Corollary \ref{corl-a} can
be laid out as follows.

\begin{exam}[$\lambda=\frac{1}{2}$ in Corollary \ref{corl-a}]
 \bnm
&&\xxqdn\frac{1}{1+2a}\:{_4F_3}\ffnk{ccccccccc}{1}{1,\frac{5-2a}{4},\frac{1}{2}+a+n,-n}
{\frac{5+2a}{4},\frac{3}{2}-a-n,2+n}\\
&&\xxqdn+\:\frac{1}{1-2a}\:{_4F_3}\ffnk{ccccccccc}{1}{1,\frac{3-2a}{4},\frac{1}{2}+a+n,-n}
{\frac{3+2a}{4},\frac{3}{2}-a-n,2+n}\\
  &&\xxqdn=\frac{2}{(1-2a)(1+2a+4n)}\frac{(2)_n(a)_n}{(\frac{1}{2}+a)_n(-\frac{1}{2}+a)_n}.
 \enm
\end{exam}

\begin{exam}[$\lambda=\frac{3}{2}$ in Corollary \ref{corl-a}]
 \bnm
&&\xxqdn\frac{1}{3+2a}\:
{_8F_7}\ffnk{ccccccccc}{1}{1,\frac{5-2a}{4},\frac{7-2a}{8},\frac{11-2a}{8},\frac{3+2a+2n}{6},\frac{5+2a+2n}{6},\frac{7+2a+2n}{6},-n}
{\frac{5+2a}{4},\frac{7+2a}{8},\frac{11+2a}{8},\frac{9-2a-2n}{6},\frac{7-2a-2n}{6},\frac{5-2a-2n}{6},2+n}\\
&&\xxqdn+\:\frac{1}{3-2a}\:
{_8F_7}\ffnk{ccccccccc}{1}{1,\frac{3-2a}{4},\frac{5-2a}{8},\frac{9-2a}{8},\frac{3+2a+2n}{6},\frac{5+2a+2n}{6},\frac{7+2a+2n}{6},-n}
{\frac{3+2a}{4},\frac{5+2a}{8},\frac{9+2a}{8},\frac{9-2a-2n}{6},\frac{7-2a-2n}{6},\frac{5-2a-2n}{6},2+n}\\
  &&\xxqdn=\frac{6}{(3-2a+4n)(3+2a+8n)}\frac{(2)_n(a)_n}{(\frac{3}{2}+a)_n(-\frac{3}{2}+a)_n}.
 \enm
\end{exam}

\begin{thm}\label{thm-c}
Let $a$ and $\lambda$ be complex numbers. Then
 \bnm
\:\:\sum_{i=0}^m\frac{(-m)_i}{i!}\Theta(\lambda,a,i,m,n)=
 \frac{(-1)^{n}2^{4a-2\lambda-m-1}(1+2a-\lambda+\frac{m}{2})_n(\lambda-2a-\frac{m}{2})_m}
 {(1+2a+2\lambda n+2n)(\lambda-a+\lambda n)(1)_n},
 \enm
where the expression on the left hand side is
 \bnm
&&\xxqdn\Theta(\lambda,a,i,m,n)\\
 &&\xxqdn\:={_8\Psi_9}\ffnk{ccccccc}{-1}
  {(1,\frac{3}{2};1),(\frac{m-i}{2}+\lambda-a;\lambda),(\frac{1+m-i}{2}+\lambda-a;1+\lambda),\\
(\frac{3+m}{2}-i+\lambda,\frac{2+m}{2}+\lambda,1+2a+n;1+2\lambda),(1-i+2\lambda;2+4\lambda)}
{(1+n;-1),(2+n,\frac{1}{2};1),(\frac{2-i}{2}+a;\lambda),(\frac{3-i}{2}+a;1+\lambda),\\
(\frac{1+m}{2}\!-\!i\!+\!\lambda,\frac{2-m}{2}\!+\!\lambda,1\!+\!2\lambda\!-\!2a\!-\!n;1\!+\!2\lambda),(2+m-i+2\lambda;2\!+\!4\lambda)}.
\enm
\end{thm}

\begin{proof}
Utilizing Lemma 2 of \cite{chu-c}, it is not difficult to obtain
 \bnm
&&\xqdn{_3F_2}\ffnk{ccccccc}{1}{a,1-a,b+\frac{m}{2}}{c,1+2b-c}=\frac{(c-2b)_m(c-b-\frac{m}{2})_m}{(2c-2b-1)_m(1-b-\frac{m}{2})_m}
\!\sum_{i=0}^m\frac{(-m)_i(\frac{3-m}{2}+b-c)_i}{i!(\frac{1-m}{2}+b-c)_i}\\
&&\xqdn\:\:\times\:\:\frac{(1-c)_i(1+2b-2c-m)_i}{(2+2b-2c)_i(1+2b-c-m)_i}
{_3F_2}\ffnk{ccccccc}{1}{a,1-a,b-\frac{m}{2}}{c-i,1+2b-c-m+i}.
  \enm
Evaluating the $_3F_2$-series on the right hand side by
\eqref{whipple} and then taking $a=-n$ in
 the resulting identity, we have
\bnm
&&\xqdn{_3F_2}\ffnk{ccccccc}{1}{-n,1+n,b+\frac{m}{2}}{c,1+2b-c}=\frac{(c-2b)_m(c-b-\frac{m}{2})_m}{(2c-2b-1)_m(1-b-\frac{m}{2})_m}
\!\sum_{i=0}^m\frac{(-m)_i(\frac{3-m}{2}+b-c)_i}{i!(\frac{1-m}{2}+b-c)_i}\\
&&\xqdn\:\:\times\:\:
\frac{(1-c)_i(1+2b-2c-m)_i}{(2+2b-2c)_i(1+2b-c-m)_i}
\frac{(\frac{c-n-i}{2})_n(c-2b+m-i)_n}{(\frac{c-n+m-i}{2}-b)_n(c-i)_n}.
\enm
 Employ the substitutions $b\to 1+2a-\lambda$ and $c\to
2+2a+2\lambda n+n$ to get
  \bnm
&&\xxqdn\sum_{k=0}^n(-1)^k\:\binm{n}{k}
 \frac{1+2a+2\lam k+k+k}{(1+2a+2\lam n+n)_{1+k}}
 \frac{2a-2\lam-2\lam k-k+k}{(2a-2\lam-2\lambda n-n)_{1+k}}(1+k)_n\\
 &&\xxqdn\:\:\times\:
  \frac{(1)_k(1+2a-\lam+\frac{m}{2})_k}{(1+2a+2\lam k+2k)(2a-2\lam-2\lam k)}\\
&&\xxqdn \:\:=\:\frac{(1)_n}{\sst(1+2a+2\lambda
n+n)(2a-2\lam-2\lambda n-n)}
 \frac{(2\lam-2a+2\lambda n+n)_m(1+\lambda+2\lambda n+n-\frac{m}{2})_m}{(1+2\lambda+4\lam n+2n)_m(\lambda-2a-\frac{m}{2})_m}\\
&&\xxqdn\:\:\times\:\sum_{i=0}^m\frac{(-m)_i(\frac{1-m}{2}-\lambda-2\lambda
n-n)_i(-1-2\lambda-4\lambda n-2n-m)_i}
{i!(-\frac{1+m}{2}-\lambda-2\lambda n-n)_i(1+2a-2\lambda-2\lambda n-n-m)_i}\\
&&\xxqdn\:\:\times\:\frac{(-1-2a-2\lambda n-n)_i}
{(-2\lambda-4\lambda n-2n)_i} \frac{(a+\lambda
n+\frac{2-i}{2})_n(2\lambda-2a+2\lambda
n+n+m-i)_n}{(\lambda-a+\lambda n+\frac{m-i}{2})_n(2+2a+2\lambda n
+n-i)_n},
 \enm
which satisfies \eqref{inver-a} with
 \bnm
&&\qdn\xqdn x=1+2a,\quad y=2a-2\lam, \quad z=1+2\lam; \\
&&\qdn\xqdn g(k)=\frac{(1)_k(1+2a-\lam+\frac{m}{2})_k}{(1+2a+2\lam
k+2k)(2a-2\lam-2\lam k)};\\
&&\qdn\xqdn f(n)=\frac{(1)_n}{\sst(1+2a+2\lambda
n+n)(2a-2\lam-2\lambda n-n)}
 \frac{(2\lam-2a+2\lambda n+n)_m(1+\lambda+2\lambda n+n-\frac{m}{2})_m}{(1+2\lambda+4\lam n+2n)_m(\lambda-2a-\frac{m}{2})_m}\\
&&\:\times\:\sum_{i=0}^m\frac{(-m)_i(\frac{1-m}{2}-\lambda-2\lambda
n-n)_i(-1-2\lambda-4\lambda n-2n-m)_i}
{i!(-\frac{1+m}{2}-\lambda-2\lambda n-n)_i(1+2a-2\lambda-2\lambda n-n-m)_i}\\
&&\:\times\:\frac{(-1-2a-2\lambda n-n)_i} {(-2\lambda-4\lambda
n-2n)_i} \frac{(a+\lambda n+\frac{2-i}{2})_n(2\lambda-2a+2\lambda
n+n+m-i)_n}{(\lambda-a+\lambda n+\frac{m-i}{2})_n(2+2a+2\lambda n
+n-i)_n}.
 \enm
Then \eqref{inver-b} produces the dual relation
 \bnm
&&\qdn\sum_{k=0}^n(-1)^k\binm{n}{k}(1+2a+2\lam k+k)_n(2a-2\lam-2\lambda k-k)_n\frac{1+2k}{(1+n)_{1+k}}\\
 &&\qdn\:\times\:\,
 \frac{(1)_k}{\sst(1+2a+2\lambda k+k)(2a-2\lam-2\lambda k-k)}
 \frac{(2\lam-2a+2\lambda k+k)_m(1+\lambda+2\lambda k+k-\frac{m}{2})_m}{(1+2\lambda+4\lam k+2k)_m(\lambda-2a-\frac{m}{2})_m}\\
&&\qdn\:\times\:\sum_{i=0}^m\frac{(-m)_i(\frac{1-m}{2}-\lambda-2\lambda
k-k)_i(-1-2\lambda-4\lambda k-2k-m)_i}
{i!(-\frac{1+m}{2}-\lambda-2\lambda k-k)_i(1+2a-2\lambda-2\lambda k-k-m)_i}\\
&&\qdn\:\times\:\frac{(-1-2a-2\lambda k-k)_i}{(-2\lambda-4\lambda
k-2k)_i}
\frac{(a+\lambda k+\frac{2-i}{2})_k(2\lambda-2a+2\lambda k+k+m-i)_k}{(\lambda-a+\lambda k+\frac{m-i}{2})_k(2+2a+2\lambda k+k-i)_k}\\
&&\qdn\:=\:\frac{(1)_n(1+2a-\lam+\frac{m}{2})_n}{(1+2a+2\lam
n+2n)(2a-2\lam-2\lam n)}.
 \enm
Interchanging the summation order, we gain Theorem \ref{thm-c} after
some simplifications.
\end{proof}

When $m=0$, Theorem \ref{thm-c} also reduces to Theorem \ref{thm-a}.
Performing the the replacements $\lambda\to \lambda-\frac{1}{2}$ and
$a\to\frac{\lambda+a-1}{2}$
 in the case $m=1$ of Theorem \ref{thm-c}, we achieve the following
reciprocal formula after some reformulations.

\begin{corl}\label{corl-b}
Let $a$ and $\lambda$ be complex numbers. Then
 \bnm
B(\lambda;a,n)+B(-\lambda;a,n)=\frac{(-1)^{n}4^{a}(a)_{1+n}}{(a+\lambda+2\lambda
n+n)(a-\lambda-2\lambda n+n)(1)_n},
 \enm
where the symbol on the left hand side stands for
 \bnm
&&\xqdn B(\lambda;a,n)\\
&&\xqdn={_5\Psi_6}\ffnk{ccccccc}{-1}
  {(1,\frac{3}{2};1),(\frac{1+\lambda-a}{2};\lambda-\frac{1}{2}),(\frac{2+\lambda-a}{2};\lambda+\frac{1}{2}),(\lambda+a+n;2\lambda)}
{(1\!+\!n;-1),(2\!+\!n,\frac{1}{2};1),(\frac{1+\lambda+a}{2};\lambda\!-\!\frac{1}{2}),(\frac{2+\lambda+a}{2};\lambda\!+\!\frac{1}{2}),(1\!+\!\lambda\!-\!a\!-\!n;2\lambda)}.
\enm
\end{corl}

Two hypergeometric series identities from Corollary \ref{corl-b} can
be displayed as follows.

\begin{exam}[$\lambda=\frac{1}{2}$ in Corollary \ref{corl-b}]
 \bnm
&&\xxqdn\frac{1}{1+2a}\:{_5F_4}\ffnk{ccccccccc}{1}{1,\frac{3}{2},\frac{5-2a}{4},\frac{1}{2}+a+n,-n}
{\frac{1}{2},\frac{5+2a}{4},\frac{3}{2}-a-n,2+n}\\
&&\xxqdn-\:\frac{1}{1-2a}\:{_5F_4}\ffnk{ccccccccc}{1}{1,\frac{3}{2},\frac{3-2a}{4},\frac{1}{2}+a+n,-n}
{\frac{1}{2},\frac{3+2a}{4},\frac{3}{2}-a-n,2+n}\\
  &&\xxqdn=\frac{4}{(2a-1)(1+2a+4n)}\frac{(1)_{1+n}(a)_{1+n}}{(\frac{1}{2}+a)_n(-\frac{1}{2}+a)_n}.
 \enm
\end{exam}

\begin{exam}[$\lambda=\frac{3}{2}$ in Corollary \ref{corl-b}]
 \bnm
&&\xxqdn\frac{1}{3+2a}\:
{_9F_8}\ffnk{ccccccccc}{1}{1,\frac{3}{2},\frac{5-2a}{4},\frac{7-2a}{8},\frac{11-2a}{8},\frac{3+2a+2n}{6},\frac{5+2a+2n}{6},\frac{7+2a+2n}{6},-n}
{\frac{1}{2},\frac{5+2a}{4},\frac{7+2a}{8},\frac{11+2a}{8},\frac{9-2a-2n}{6},\frac{7-2a-2n}{6},\frac{5-2a-2n}{6},2+n}\\
&&\xxqdn-\:\frac{1}{3-2a}\:
{_9F_8}\ffnk{ccccccccc}{1}{1,\frac{3}{2},\frac{3-2a}{4},\frac{5-2a}{8},\frac{9-2a}{8},\frac{3+2a+2n}{6},\frac{5+2a+2n}{6},\frac{7+2a+2n}{6},-n}
{\frac{1}{2},\frac{3+2a}{4},\frac{5+2a}{8},\frac{9+2a}{8},\frac{9-2a-2n}{6},\frac{7-2a-2n}{6},\frac{5-2a-2n}{6},2+n}\\
  &&\xxqdn=\frac{4}{(2a-3-4n)(3+2a+8n)}\frac{(1)_{1+n}(a)_{1+n}}{(\frac{3}{2}+a)_n(-\frac{3}{2}+a)_n}.
 \enm
\end{exam}

\section{Two different summation formulas for Fox-Wright function}
\begin{thm}\label{thm-d}
Let $a$, $b$ and $c$ be complex numbers. Then
 \bnm
&&\xqdn{_6\Psi_7}\ffnk{ccccccc}{-1}
{(1,\frac{3}{2};1),(\frac{1}{2}-a+b,a+n;\frac{1}{2}),(b-a,-\frac{1}{2}+a+n;-\frac{1}{2})}
{(1\!+\!n;-1),(2+n,\frac{1}{2};1),(1\!+\!a\!-\!c,\frac{1}{2}\!-\!a\!+\!b\!+\!c;\frac{1}{2}),(\frac{1}{2}\!+\!a\!-\!c,b\!+\!c\!-\!a;-\frac{1}{2})}\\
 &&\xqdn\:\:=\:
\frac{2(\frac{1}{2}+2a-b-c)_n(b)_n(c)_n}{(1)_n(1+2a-2b)_n}\frac{\Gamma(2a-1+n)\Gamma(2b-2a)}{\Gamma(1+2a-2c+n)\Gamma(2b+2c-2a+n)}.
 \enm
\end{thm}

\begin{proof}
A $_7F_6$-series identity due to \cite{chu-e}[Corollary 16] can be
expressed as
 \bnm
&&\xxqdn{_7F_6}\ffnk{ccccccccc}{1}{a,1+\frac{a}{3},b,c,\frac{1}{2}+a-b-c,1+n,-n}
{\frac{a}{3},1+a-2b,1+a-2c,2b+2c-a,\frac{1+a-n}{2},\frac{2+a+n}{2}}\\
&&\xxqdn\:\:=\:\frac{(1+a)_n(2b-a)_n(\frac{1+a-n}{2}-c)_n(b+c-\frac{a+n}{2})_n}{(1+a-2c)_n(2b+2c-a)_n(\frac{1+a-n}{2})_n(b-\frac{a+n}{2})_n}.
\enm
 Employ the substitution $a\to 2a$ to attain
 \bnm
&&\sum_{k=0}^n(-1)^k\:\binm{n}{k}
 \frac{a+\frac{k}{2}+k}{(a+\frac{n}{2})_{1+k}}
  \frac{a-\frac{1}{2}-\frac{k}{2}+k}{(a-\frac{1}{2}-\frac{n}{2})_{1+k}}(1+k)_n\\
 &&\:\:\times\:
  \frac{(1)_k(2a-1)_k(b)_k(c)_k(\frac{1}{2}+2a-b-c)_k}{(1+2a-2b)_k(1+2a-2c)_k(2b+2c-2a)_k}\\
&& \:\:=\:\frac{(1)_n(2b-2a)_n(2a-1)_n}{(1+2a-2c)_n(2b+2c-2a)_n}
 \frac{(\frac{1-n}{2}+a-c)_n(b+c-a-\frac{n}{2})_n}{(a-\frac{1+n}{2})_n(b-a-\frac{n}{2})_n},
 \enm
which fits to \eqref{inver-a} with
 \bnm
&&\qquad x=a,\quad y=a-\frac{1}{2}, \quad z=\frac{1}{2}; \\
&&\qquad g(k)= \frac{(1)_k(2a-1)_k(b)_k(c)_k(\frac{1}{2}+2a-b-c)_k}{(1+2a-2b)_k(1+2a-2c)_k(2b+2c-2a)_k};\\
&&\qquad f(n)=\frac{(1)_n(2b-2a)_n(2a-1)_n}{(1+2a-2c)_n(2b+2c-2a)_n}
 \frac{(\frac{1-n}{2}+a-c)_n(b+c-a-\frac{n}{2})_n}{(a-\frac{1+n}{2})_n(b-a-\frac{n}{2})_n}.
 \enm
Then \eqref{inver-b} gives the dual relation
 \bnm
&&\qdn\sum_{k=0}^n(-1)^k\binm{n}{k}\bigg(a+\frac{k}{2}\bigg)_n
\bigg(a-\frac{1}{2}-\frac{k}{2}\bigg)_n\frac{1+2k}{(1+n)_{1+k}}\\
 &&\qdn\:\times\:\,
\frac{(1)_k(2b-2a)_k(2a-1)_k}{(1+2a-2c)_k(2b+2c-2a)_k}
 \frac{(\frac{1-k}{2}+a-c)_k(b+c-a-\frac{k}{2})_k}{(a-\frac{1+k}{2})_k(b-a-\frac{k}{2})_k}\\
&&\qdn\:=\:\frac{(1)_n(2a-1)_n(b)_n(c)_n(\frac{1}{2}+2a-b-c)_n}{(1+2a-2b)_n(1+2a-2c)_n(2b+2c-2a)_n}.
 \enm
Writing it in terms of Fox-Wright function, we obtain Theorem
\ref{thm-d}.
\end{proof}

When $b\to\infty$, Theorem \ref{thm-d} offers the following formula.

\begin{corl}\label{corl-c}
Let $a$ and $c$ be complex numbers. Then
 \bnm
&&\xqdn{_4\Psi_5}\ffnk{ccccccc}{-1}
{(1,\frac{3}{2};1),(a+n;\frac{1}{2}),(-\frac{1}{2}+a+n;-\frac{1}{2})}
{(1+n;-1),(2+n,\frac{1}{2};1),(1+a-c;\frac{1}{2}),(\frac{1}{2}+a-c;-\frac{1}{2})}\\
 &&\xqdn\:\:=\:
\frac{1}{2^{2c+2n-1}}\frac{(c)_n}{(1)_n}\frac{\Gamma(2a-1+n)}{\Gamma(1+2a-2c+n)}.
 \enm
\end{corl}

Fixing $c=1-m-n$ in Corollary \ref{corl-c}, we get the following
identity.

\begin{corl}\label{corl-d}
Let $a$ be a complex number and $m$ a nonnegative integer. Then
 \bnm
&&\xqdn{_{3+2m}F_{2+2m}}\ffnk{ccccccc}{1}
{1,\frac{3}{2},\:\{2a+2n+2i-2\}_{i=1}^m,\:\{3-2a-2n-2i\}_{i=1}^m,\:-n}
{\frac{1}{2},\:\{2a+2n+2i-1\}_{i=1}^m,\:\{4-2a-2n-2i\}_{i=1}^m,\:2+n}\\
 &&\xqdn\:\:=\:
(-1)^n\frac{(2)_n(m)_n}{(2a-1+n)_n(2a-1+2m+2n)_n}.
 \enm
\end{corl}

\begin{thm}\label{thm-e}
Let $a$ and $b$ be complex numbers. Then
 \bnm
&&\xxqdn\xqdn{_7\Psi_6}\ffnk{ccccccc}{\frac{1}{4}}
{(1,\frac{3}{2},1-a-b,b-a,a+n;1),(\frac{1-a-b}{2},\frac{b-a}{2};-\frac{1}{2})}
{(1+n;-1),(2+n,2-a-n,\frac{1}{2};1),(\frac{1-a-b}{2},\frac{b-a}{2};\frac{1}{2})}\\
 &&\xxqdn\xqdn\:\:=\:
\frac{(b)_n(1-b)_n}{(a+b)_n(1+a-b)_n}\frac{\Gamma(a+n)\Gamma(1-a-b)\Gamma(b-a)}{2\Gamma(1+n)\Gamma(2-a)}.
 \enm
\end{thm}

\begin{proof}
Setting $d=1-b$ and $e=1-c$ in the transformation formula (cf.
\cite{andrews-r}[p. 147])
 \bnm
 &&\qqdn\xxqdn{_6F_5}\ffnk{ccccccc}{-1}{a,&1+a/2,&b,&c,&d,&e}
 {&a/2,&1+a-b,&1+a-c,&1+a-d,&1+a-e} \\
&&\xxqdn=\frac{\Gamma(1+a-b)\Gamma(1+a-c)}{\Gamma(1+a)\Gamma(1+a-b-c)}
{_3F_2}\!\ffnk{cccc}{1}{1+a-d-e,\quad b,\quad c}
{\quad1+a-d,\quad1+a-e}
 \enm
 and calculating the series on the right hand side by
 Dixon's $_3F_2
     $-series identity(cf. \cite{andrews-r}[p. 72]):
 \bnm
 &&\xxqdn\xxqdn\qqdn_3F_2\ffnk{cccc}{1}{a,b,c}{1+a-b,1+a-c}\\
&&\xxqdn\xxqdn\qqdn\:\:=\:\frac{\Gamma(1+\frac{a}{2})\Gamma(1+a-b)\Gamma(1+a-c)\Gamma(1+\frac{a}{2}-b-c)}
{\Gamma(1+a)\Gamma(1+\frac{a}{2}-b)\Gamma(1+\frac{a}{2}-c)\Gamma(1+a-b-c)},
 \enm
where $Re(1+\frac{a}{2}-b-c)>0$, we gain
 \bnm
 &&\qqdn{_6F_5}\ffnk{ccccccc}{-1}{a,1+\frac{a}{2},b,1-b,c,1-c}
 {\frac{a}{2},1+a-b,a+b,1+a-c,a+c}\\
&&=\frac{\pi\Gamma(a+b)\Gamma(1+a-b)\Gamma(a+c)\Gamma(1+a-c)}
{2^{2a-1}\Gamma(a)\Gamma(1+a)\Gamma(\frac{a+b+c}{2})\Gamma(\frac{1+a+b-c}{2})\Gamma(\frac{1+a-b+c}{2})\Gamma(\frac{2+a-b-c}{2})}
 \enm
provided that $Re(a)>0$. The case $c=-n$ of it can be manipulated as
 \bnm
&&\xqdn\sum_{k=0}^n(-1)^k\binm{n}{k}\frac{a+k+k}{(a+n)_{1+k}}
  \frac{a-1-k+k}{(a-1-n)_{1+k}}(1+k)_n
  \frac{(1)_k(a)_k(b)_k(1-b)_k}{(1+a-b)_k(a+b)_k}(-1)^k\\
&&\xqdn\:\:=\:\frac{(1)_n(a)_n(1-a-b)_n(b-a)_n}{(1+a)_n(-a)_n(2-a)_n}
 \frac{(-\frac{a+n}{2})_n(\frac{1-a-n}{2})_n}{(\frac{1-a-b-n}{2})_n(\frac{b-a-n}{2})_n},
 \enm
which suits to \eqref{inver-a} with
 \bnm
&&\xxqdn x=a,\quad y=a-1, \quad z=1; \\
&&\xxqdn g(k)=\frac{(1)_k(a)_k(b)_k(1-b)_k}{(1+a-b)_k(a+b)_k}(-1)^k;\\
&&\xxqdn
f(n)=\frac{(1)_n(a)_n(1-a-b)_n(b-a)_n}{(1+a)_n(-a)_n(2-a)_n}
 \frac{(-\frac{a+n}{2})_n(\frac{1-a-n}{2})_n}{(\frac{1-a-b-n}{2})_n(\frac{b-a-n}{2})_n}.
 \enm
Then \eqref{inver-b} produces the dual relation
 \bnm
&&\xxqdn\sum_{k=0}^n(-1)^k\binm{n}{k}(a+k)_n(a-1-k)_n\frac{1+2k}{(1+n)_{1+k}}\frac{(1)_k(a)_k(1-a-b)_k(b-a)_k}{(1+a)_k(-a)_k(2-a)_k}\\
 &&\xxqdn\:\times\:\,
 \frac{(-\frac{a+k}{2})_k(\frac{1-a-k}{2})_k}{(\frac{1-a-b-k}{2})_k(\frac{b-a-k}{2})_k}
=\frac{(1)_n(a)_n(b)_n(1-b)_n}{(1+a-b)_n(a+b)_n}(-1)^n.
 \enm
Writing it according to Fox-Wright function, we achieve Theorem
\ref{thm-e}.
\end{proof}

The case $d\to\infty$ of \eqref{dougall} reads
 \bnm\qquad\quad
 _4F_3\ffnk{cccc}{-1}{a,1+\frac{a}{2},b,c}{\frac{a}{2},1+a-b,1+a-c}
=\frac{\Gamma(1+a-b)\Gamma(1+a-c)}{\Gamma(1+a)\Gamma(1+a-b-c)}
 \enm
with $Re(1+\frac{a}{2}-b-c)>0$. When $b\to\infty$, Theorem
\ref{thm-e} reduces to the special case of it:
 \bnm\qquad
 _4F_3\ffnk{cccc}{-1}{1,\frac{3}{2},a+n,-n}{\frac{1}{2},2-a-n,2+n}
=(-1)^n\frac{(2)_n}{(a-1)_n}.
 \enm

Taking $b=-2m$ in Theorem \ref{thm-e}, we attain the following
result.

\begin{corl}\label{corl-e}
Let $a$ be a complex number and $m$ a nonnegative integer. Then
 \bnm
&&\xqdn{_{5+4m}F_{4+2m}}\ffnk{ccccccc}{1}
{\sst1,\frac{3}{2},\:\{1-a+2i\}_{i=0}^m,\:\{a+2i\}_{i=1}^m,\:\{1-a-2i\}_{i=1}^m\:,\{2+a-2i\}_{i=1}^m,\:a+n,-n}
{\sst\frac{1}{2},\:\{1+a+2i\}_{i=0}^m,\:\{-a+2i\}_{i=1}^m,\:\{1+a-2i\}_{i=1}^m,\:\{2-a-2i\}_{i=1}^m,\:2-a-n,2+n}\\
 &&\xqdn\:\:=\:
(-1)^n\frac{(-2m)_n(1+2m)_n(2)_n}{(a-1)_n(a-2m)_n(1+a+2m)_n}.
 \enm
\end{corl}

\acknowledgements

The authors are grateful to the reviewer for helpful comments.

\nocite{*}
\bibliographystyle{abbrvnat}
\bibliography{sample-dmtcs}
\label{sec:biblio}


\end{document}